\documentclass[12pt]{article}
\usepackage[cm]{fullpage}
\usepackage[small]{titlesec}
\usepackage[cm]{fullpage}

\usepackage{amsmath,amscd, float,rotating}
\usepackage{pb-diagram}

\usepackage{latexsym}
\usepackage{amsfonts,color}
\usepackage{amsmath}
\usepackage{amssymb}

\numberwithin{equation}{section}
\newcommand{\qed}{\hfill \ensuremath{\Box}}

\def\XXint#1#2#3{{\setbox0=\hbox{$#1{#2#3}{\int}$}
\vcenter{\hbox{$#2#3$}}\kern-.5\wd0}}

\newcommand{\dbar}{\overline{\partial}}

\newcommand{\ddbar}{\frac{\sqrt{-1}}{2\pi} \partial\dbar}

\begin{document}
\newcounter{remark}
\newcounter{theor}
\setcounter{remark}{0} \setcounter{theor}{1}
\newtheorem{claim}{Claim}
\newtheorem{theorem}{Theorem}[section]
\newtheorem{proposition}{Proposition}[section]
\newtheorem{lemma}{Lemma}[section]
\newtheorem{definition}{Definition}[section]
\newtheorem{conjecture}{Conjecture}[section]
\newtheorem{corollary}{Corollary}[section]
\newenvironment{proof}[1][Proof]{\begin{trivlist}
\item[\hskip \labelsep {\bfseries #1}]}{\end{trivlist}}
\newenvironment{remark}[1][Remark]{\addtocounter{remark}{1} \begin{trivlist}
\item[\hskip \labelsep {\bfseries #1
\thesection.\theremark}]}{\end{trivlist}}
\newenvironment{example}[1][Example]{\addtocounter{remark}{1} \begin{trivlist}
\item[\hskip \labelsep {\bfseries #1
\thesection.\theremark}]}{\end{trivlist}}
~

\begin{center}
{\large \bf
Compactness of K\"ahler-Ricci solitons on Fano manifolds
\footnote{Work supported in part by
National Science Foundation grants DMS-1711439, DMS-12-66033 and DMS-1710500. }}
\bigskip\bigskip

{Bin Guo$^*$, Duong H. Phong$^{**}$,  Jian Song$^\dagger$ and Jacob Sturm$^{\ddagger}$} \\

\bigskip

\end{center}

\begin{abstract}

{\footnotesize In this short paper, we improve the result of Phong-Song-Sturm  on degeneration of Fano K\"ahler-Ricci solitons by removing the assumption on the uniform bound of the Futaki invariant. Let  $\mathcal{KR}(n)$ be the space of  K\"ahler-Ricci solitons on $n$-dimensional  Fano manifolds. We show that  after passing to a subsequence, any sequence in $\mathcal{KR}(n)$ converge in the Gromov-Hausdorff topology to a  K\"ahler-Ricci soliton on an $n$-dimensional $\mathbb{Q}$-Fano variety with log terminal singularities. }

\end{abstract}

\section{Introduction}  

The Ricci solitons on compact and complete Riemannian manifolds naturally arise as models of singularities for the Ricci flow \cite{H}. The existence and uniqueness of Ricci solitons has been extensively studied. A gradient  Ricci soliton is a Riemannian metric satisfying the following soliton equation
\begin{equation}\label{soleq1}
Ric(g) = \lambda g + \nabla^2 u
\end{equation}
for some smooth function $f$ with $\lambda = -1, 0, 1$.  Such a soliton is called a gradient shrinking Ricci soliton if  $\lambda>0$. If we let the vector field $\mathcal{V}$ be defined by $\mathcal{V}= \nabla u$, the soliton equation becomes
\begin{equation}\label{soleq2}
 Ric(g) = \lambda g + L_{\mathcal{V}} g, 
\end{equation}
where $L_{\mathcal{V}}$ is the Lie derivative along $\mathcal{V}$.

A K\"ahler metric $g$ on a K\"ahler manifold $X$ is called a K\"ahler-Ricci soliton if it satisfies the soliton equation (\ref{soleq1}) or equation (\ref{soleq2})  for  $\mathcal{V}=\nabla u$. Any shrinking K\"ahler-Ricci soliton on a compact K\"ahler manifold $X$ must be a gradient Ricci soliton and such a K\"ahler manifold must be a Fano manifold, i.e. $c_1(X)>0$. The vector field $\mathcal V$ must be holomorphic and it can be expressed in terms of the Ricci potential $u$, with
\begin{equation}
\label{soliton}
R_{i\bar j}= g_{i\bar j} -  u_{i\bar j}, ~~u_{ij}=  u_{\bar i \bar j}= 0, ~\mathcal{V}^i = -g^{i\bar j}u_{\bar j}. 
\end{equation}

The well-known Futaki invariant associated to the K\"ahler-Ricci soliton $(X, g, \mathcal{V})$ on a Fano manifold $X$ is given by
$$
\mathcal{F}_X(\mathcal{V}) = \int_X |\nabla u|^2 dV_g = \int_X |\mathcal{V}|^2 dV_g \geq 0. 
$$

Let $\mathcal{KR}(n, F)$ be the set of compact K\"ahler-Ricci solitons $(X, g)$ of complex dimension $n$ with $$Ric(g) = g + L_{\mathcal{V}} g,~  \mathcal{F}_X (\mathcal{V}) \leq F.$$
It is proved by Tian-Zhang \cite{TZ} that $\mathcal{KR}(n, F)$ is compact in the Gromov-Hausdorff topology with an additional uniform upper volume bound. In \cite{PSS}, Phong-Song-Sturm established a partial $C^0$-estimate on $\mathcal{KR}(n, F)$, generalizing the celebrated result of Donaldson-Sun \cite{DS} for the space of uniformly non-collapsed K\"ahler manifolds with uniform Ricci curvature bounds. An immediate consequence of the partial $C^0$-estimate in \cite{PSS} is that the limiting metric space must be a $\mathbb{Q}$-Fano variety equipped with a K\"ahler-Ricci soliton metric. 

The purpose of this paper is to remove the assumption in \cite{PSS} on the bound of the Futaki invariant. \begin{definition}
Let $\mathcal{KR}(n)$ be the set of   compact K\"ahler-Ricci solitons $(X, g, \mathcal{V})$ of complex dimension $n$ with $$Ric(g) = g + L_{\mathcal{V}} g.$$
\end{definition}

The following is the main result of the paper.

\begin{theorem} \label{main1} Let $\{ (X_i, g_i, \mathcal{V}_i) \}_{i=1}^\infty $ be a sequence in  $\mathcal{KR}(n)$ with $n\geq 2$. Then after possibly passing to subsequence, $(X_i, g_i)$ converges in the Gromov-Hausdorff topology to a compact metric length space $(X_\infty, d_\infty)$ satisfying the following.

\begin{enumerate}

\item The singular set $\Sigma_\infty$ of the metric space $(X_\infty, d_\infty)$ is a closed set of Hausdorff dimension no greater than $2n-4$. 

\item $(X_i, g_i, \mathcal{V}_i)$ converges smoothly to a K\"ahler-Ricci soliton  $(X_\infty\setminus \Sigma_\infty, g_\infty, \mathcal{V}_\infty)$ satisfying 
\begin{equation}
Ric(g_\infty) = g_\infty + L_{\mathcal{V}_\infty} g_\infty,
\end{equation}
where $\mathcal{V}_\infty$ is a holomorphic vector field on $X_\infty\setminus \Sigma_\infty$.  

\item $(X_\infty, d_\infty)$ coincides with the metric completion of $(X_\infty \setminus \Sigma_\infty, g_\infty)$ and it is a projective $\mathbb{Q}$-Fano variety with log terminal singularities. The soliton K\"ahler metric $g_\infty$ extends to a K\"ahler current on $X_\infty$ with bounded local potential and $\mathcal{V}_\infty$ extends to a global holomorphic vector field on $X_\infty$.


\end{enumerate}

\end{theorem}

The assumption on the bound of the Futaki invariant in \cite{PSS} is used to obtain a uniform lower bound of Perelman's $\mu$-functional. We use the recent deep result of Birkar \cite{B} in birational geometry and show that there exists $\epsilon(n)>0$ such that for any $n$-dimensional Fano manifold $X$, there exists a K\"ahler metric $g$ with $Ric(g) \geq \epsilon g.$ In particular, the $\mu$-functional for $(X, g)$ is bounded below by a uniform constant that only depends on $n$. 
Then for any K\"ahler-Ricci soliton $(X, g)\in \mathcal{KR}(n)$, the $\mu$-functional for $(X, g)$ is uniformly bounded below because the soliton metric is the limit of the K\"ahler-Ricci flow.   
The proof of Theorem \ref{main1} also implies a uniform bound for the scalar curvature and the Futaki invariant for all $(X, g)\in \mathcal{KR}(n)$.

\begin{corollary} \label{main2} There exist $F=F(n)$, $D=D(n)$ and $K=K(n)>0$  such that 
for any $(X, g, u) \in \mathcal{KR}(n)$, the Futaki invariant, the diameter and scalar curvature $R$ of $(X, g)$ satisfy
$$\mathcal{F}_X \leq F,~ diam(X, g) \leq D,  ~ 0< R \leq K. $$

\end{corollary}

We also derive some general compactness for compact or complete gradient shrinking solitons assuming a uniform lower bound of Perelman's $\mu$-functional (see Section 3). For any closed or complete gradient shrinking soliton $(M, g, u)$, one can always normalize $u$ such that $\int_M e^{-u} dV_g =1$. We define $\mathcal{RS}(n, A)$ to be the space of  closed or complete shrinking gradient soliton $(M, g, u)$ of real dimension $n\geq 4$ satisfying
\begin{equation}
\mu(g) \geq -A.
\end{equation}
 Then for any $A\geq 0$ and any sequence $(M_j, g_j, u_j , p_j) \in \mathcal{RS}(n, A)$ with $p_j$ being the minimal point of $u_j$, after passing to a subsequence, it converges in the pointed Gromov-Hausdorff topology to a compact or complete metric space $(M_\infty, d_\infty)$ of dimension $n$ with smooth convergence to a shrinking gradient Ricci soliton outside the closed singular set of  dimension no greater than $n-4$.


\section{Proof of Theorem \ref{main1}}


Let us first recall the $\alpha$-invariant introduced by Tian on a Fano manifold \cite{T0}. 

\begin{definition}
On a Fano manifold $(X,\omega)$ with $\omega\in c_1(X)$, the $\alpha$-invariant is defined as
$$\alpha(X)= \sup\{\alpha>0~|~\exists C_\alpha<\infty \text{ such that  }\int_X e^{-\alpha(\varphi - \sup_X \varphi)} \omega^n \le C_\alpha, \, \forall \varphi\in PSH(X,\omega)\}.$$

\end{definition}
It is obvious that the  $\alpha(X)$ does not depend on the choice $\omega \in c_1(X)$.

\begin{definition}
Let $X$ be a normal projective variety and $\Delta$ an effective $\mathbb Q$-Cartier divisor, the pair $(X,\Delta)$ is said to be log canonical if the coefficients of components of $\Delta$ are no greater than $1$ and there exists a log resolution $\pi:Y\to X$ such that $\pi^{-1}(\mathrm{supp} \Delta )\cup \mathrm{exc}(\pi)$ is a divisor with normal crossings satisfying 
$$K_Y = \pi^*(K_X + \Delta) + \sum_j a_j F_j,\quad  \mathbb Q \ni a_j\ge -1,\,\forall\; j.$$

\end{definition}

\begin{definition}
Let $X$ be a projective manifold and $D$ be a $\mathbb Q$-Cartier divisor. The {\em log canonical threshold} of $D$ is defined  by
$$\mathrm{lct}(X,D) = \sup\{t\in\mathbb R ~|~ (X,t D) \text{ is log canonical}\}.$$

\end{definition}
It is proved by Demailly that the $\alpha$-invariant is related to the log canonical thresholds of anti-canonical divisors through the following formula (see Theorem A.3. in the Appendix A of \cite{CS}). 
\begin{theorem} For any Fano manifold $X$, 
$$\alpha (X) = \inf_{m\in \mathbb Z_{>0}} \inf_{D\in |-mK_X|} \mathrm{lct}(X, m^{-1} D)$$
\end{theorem}

Recently Birkar (Theorem 1.4 of \cite{B}) obtains a uniform positive lower bound of the log canonical threshold and the following is an immediate corollary of Birkar's result. 
\begin{theorem}
There exists   $\varepsilon_0=\varepsilon_0(n)>0$  such that for any $n$-dimensional Fano manifold $X$ 
$$\alpha(X)\ge \varepsilon_0(n).$$
\end{theorem}


From the Harnack inequality in \cite{T0}, for any fixed K\"ahler metric $\omega\in c_1(X)$,  the curvature equation for $\omega_t$ along the continuity method 
\begin{equation}\label{cont1}
Ric(\omega_t) =t(\omega_t) + (1-t)\omega
\end{equation}  can be solved for all $t\in [0, (n+1)\alpha(X)/n)$. As a consequence, we have the following corollary.

\begin{corollary} There exists $\varepsilon_1= \varepsilon_1(n)>0$ such that for any  $n$-dimensional Fano manifold $X$, there exists a K\"ahler metric $\hat\omega\in c_1(X)$ satisfying 
\begin{equation}\label{eqn:omega0}Ric(\hat\omega)\ge \varepsilon_1 \hat\omega.
\end{equation}

\end{corollary}

We can also assume that $\hat\omega$ is invariant under the group action of  the maximal compact subgroup $G$ of $Aut(X)$ by choosing a $G$-invariant K\"ahler metric $\omega$ in the equation (\ref{cont1}).

The greatest Ricci lower bound $R(X)$ for a Fano manifold $X$ is introduced in \cite{T2, Sz} and is defined by
$$\mathcal{R}(X) = \sup\{t \in\mathbb R~|~ \exists\; \omega\in c_1(X) \text{ such~that }Ric(\omega)\ge t \omega\}. $$
Immediately one has the following corollary. 

\begin{corollary} \label{cor:2.1}There exists an $r_0=r_0(n)>0$ such that for any $n$-dimensional Fano manifold $X$, 
\begin{equation}
\mathcal{R}(X) \ge  r_0.
\end{equation}
\end{corollary}

We are informed by Xiaowei Wang that Corollary \ref{cor:2.1} is already a consequence of  results in \cite{LWX}.  In fact, Theorem 5.2 and Proposition 5.1 in \cite{LWX} will imply that there exist $m=m(n)>0$ and $\beta=\beta(n) \in (0, 1]$ such that  for any $n$-dimensional Fano manifold $X$, there exists  a smooth divisor $D \in |-mK_X|$ and a conical K\"ahler-Einstein metric $\omega \in c_1(X)$ satisfying $$Ric(\omega) = \beta \omega + (1-\beta)m^{-1} [D].$$
Then Corollary \ref{cor:2.1} immediately follows by the relation between $\mathcal{R}(X)$ and the existence of conical K\"ahler-Einstein metric established in \cite{SW}.



Now let us recall Perelman's entropy functional for a Fano manifold $(X, g)$ with the associated K\"ahler form $\omega_g \in c_1(X)$. The $\mathcal W$-functional is defined by 
\begin{equation*}
\mathcal{W}(g, f) =\frac{1}{V} \int_X (R + |\nabla f|^2 + f - 2 n)  e^{-f} dV_g,
\end{equation*}
where $V= c_1^n(X)$, and the $\mu$-functional is defined by 
\begin{equation*}
\mu(g) = \inf_f\left\{ \mathcal{W}(g, f)~\left|~ \frac{1}{V} \int_X e^{-f} d V_g = 1 \right. \right\}. 
\end{equation*}
\begin{lemma}\label{lemma2.0}
There exists  $A= A(n)>0$ such that for the Riemannian metric $\hat g$ associated to the form $\hat\omega$ in \eqref{eqn:omega0}
$$\mu(\hat g)\ge -A.$$
\end{lemma}
\begin{proof}
Since $Ric(\hat g)$ is bounded from below by a uniform positive constant $\varepsilon_1(n)$, by Myers' theorem and volume comparison, 
$$\mathrm{Vol}(X,\hat g)\le C(n),\quad \mathrm{diam}(X,\hat g)\le C(n).$$
On the other hand, since $\hat \omega\in c_1(X)$ is in an integral cohomology class, in particular $\mathrm{Vol}(X,\hat g)\ge c(n)>0$. By Croke's theorem, the Sobolev constant $C_S$ of $(X,\hat g)$ is uniformly bounded. It is well-known that a Sobolev inequality implies the lower bound of $\mu$-functional. For completeness, we provide a proof below. 

For any $f\in C^\infty$ with $\int_X e^{-f}dV_{\hat g} = V$, we write $e^{-f/2} = \phi$.  By Jensen's inequality
\begin{align*}
\frac 1 V\int_X \phi^2 \log \phi^{\frac{2}{n-1}} & \le \log \Big( \frac 1 V\int_X \phi ^{\frac{2n}{n-1}}  \Big)\\
&\le \log  \Big( C_S \int_X ( |\nabla \phi|^2 + \phi^2  )    \Big)\\
& \le \frac{4}{n-1}\int_X |\nabla \phi|^2 + C(n).
\end{align*}
So
\begin{align*}
\mathcal W(\hat g,f) & = \frac 1 V\int_X ( R \phi^2 + 4 |\nabla\phi|^2 - \phi^2 \log \phi^2  )dV_{\hat g} - 2 n\ge -C(n).
\end{align*}
\qed
\end{proof}

Let $(X,g)\in\mathcal{KR}(n)$ be a gradient shrinking K\"ahler-Ricci soliton which satisfies the equation
\begin{equation}\label{eqn:KRS}Ric(\omega_g) +\ddbar u = \omega_g, \quad  \nabla  \nabla u = 0.\end{equation}
Let $G\subset Aut(X)$ be the compact one-parameter subgroup generated by the holomorphic vector field $\mathrm{Im}(\nabla u)$. As we mentioned before, the metric $\hat\omega$ in \eqref{eqn:omega0} can be taken to be $G$-invariant. 

\begin{corollary} For any $(X, g) \in \mathcal{KR}(n)$, we have 
\begin{equation*}
\mu(g) \geq -A,
\end{equation*}where $A = A(n)$ is the constant in Lemma \ref{lemma2.0}.
\end{corollary}

\begin{proof}
We consider the normalized K\"ahler-Ricci flow with initial metric $\hat\omega$ in \eqref{eqn:omega0} $G$-invariant. 
\begin{equation*}
\frac{\partial \omega(t)}{\partial t} = - Ric(\omega(t)) + \omega(t),\quad \omega(0) = \hat \omega.
\end{equation*}
By the convergence theorem for K\"ahler-Ricci flow (\cite{TZ, TZZZ}), $\omega(t)$ converges smoothly to $\omega_g$, modulo some diffeomorphisms. So $\lim_{t\to \infty} \mu( g(t) ) = \mu(g)$. 

On the other hand, $\mu(g(t))$ is monotonically non-decreasing along the K\"ahler-Ricci flow (\cite{P}). The lower bound of $\mu(g)$ follows from this monotonicity and the lower bound of $\mu(\hat g)$ established in Lemma \ref{lemma2.0}.

\qed\end{proof}



Now we can apply the same argument as in \cite{PSS} because the assumption of the uniform bound for the Futaki invariant in \cite{PSS} is to obtain a uniform lower bound for the $\mu$-functional. This will complete the proof of Theorem \ref{main1}. The argument in \cite{PSS} also implies the uniform bound for the scalar curvature and diameter of $(X, g, u)\in \mathcal{KR}(n)$ as well as $|\nabla u|^2$ and hence the Futaki invariant of $(X, g)$.  This implies Corollary \ref{main2}.




\section{Generalizations}

We generalize our previous discussion to Riemannian complete gradient shrinking Ricci solitons $(M^n ,g ,u)$ satisfying the equation
$$Ric(g) + \nabla^2 u = \frac  1 2 g.$$
By \cite{CZ} we can always normalize $u$ such  that $\int_M e^{-u}dV_g = 1$.
\begin{definition}
We denote $\mathcal{RS}(n, A)$ to the set of $n$-dimensional closed or complete shrinking gradient Ricci solitons $(M, g, u)$ satisfying
\begin{equation*}
\mu(g) \ge -A
\end{equation*}
with the normalization condition $\int_M e^{-u} dV_g = 1$. 
\end{definition}

The following proposition is the main result of this section and most results in the proposition are straightforward applications of the compactness results  \cite{WZ,ZZh} with Bakry-Emery Ricci curvature bounded below. 

\begin{proposition} \label{propgen} Let $\{ (M_i, g_i, u_i, p_i) \}_{i=1}^\infty $ be a sequence in  $\mathcal{RS}(n, A)$ with $n\geq 4$, where $p_i$ be a minimal point of $u_i$. Then after possibly passing to subsequence, $(M_i, g_i, u_i, p_i)$ converges in the Gromov-Hausdorff topology to a metric length space $(M_\infty, d_\infty, u_\infty)$ satisfying the following.
\begin{enumerate}

\item The singular set $\Sigma_\infty$ of the metric space $(M_\infty, d_\infty)$ is a closed set of Hausdorff dimension no greater than $n-4$. 

\item $(M_i, g_i, u_i)$ converges smoothly to a gradient shrinking Ricci soliton  $(M_\infty\backslash \Sigma_\infty, g_\infty, u_\infty)$ satisfying 
\begin{equation*}
Ric(g_\infty) =\frac 1 2 g_\infty + \nabla^2  u_\infty.
\end{equation*}
\item $(M_\infty, d_\infty)$ coincides with the metric completion of $(M_\infty \setminus \Sigma_\infty, g_\infty)$.

\end{enumerate}
Furthermore, if there exists $V>0$ such that $Vol_{g_i}(M_i) \leq V$ for all $i=1, 2, ...$, the limiting metric space $(M_\infty, d_\infty)$ is compact.

\end{proposition}

\begin{proof} For any $(M,g,u)\in\mathcal{RS}(n,A)$, $R = n/2 - \Delta u\ge 0$ (\cite{Zz}), the potential function $u$ satisfies $$\Delta u - |\nabla u|^2 + u = a,\quad a = \int_M u e^{-u}dV_g.$$
We denote $\tilde u = u -a$.
From $\Delta u\le n/2$ and immediately  we have $|\nabla \tilde u|^2\le n/2 +\tilde u$. By \cite{CZ}, the minimum of $\tilde u$ is achieved at some finite point $p\in M$, so $\min \tilde u = \tilde u(p)\ge -n/2$. Applying maximum principle to $\tilde u$ which satisfies $\Delta \tilde u - |\nabla \tilde u|^2 + \tilde u = 0$ at a minimum point $p\in M$, we obtain that $\min_M \tilde u = \tilde u(p)\le 0$.

From $|\nabla \tilde u|^2 \le \tilde u + n/2$, we have $|\nabla \sqrt{\tilde u + n/2}|\le \frac 1 2$. Thus for any $x\in M$
\begin{equation}\label{eqn:u}
\tilde u(x)\le \frac  12 d(p,x)^2 + \tilde u(p) + C(n)\le \frac 1 2 d(p,x)^2 + C(n).
\end{equation}
Immediately we have
\begin{equation}\label{eqn:grad u}
|\nabla \tilde u|^2(x)\le \frac 1 2 d(p,x)^2 + C(n),
\end{equation}
and
\begin{equation}\label{eqn:lap u}
-n/2\le -\Delta \tilde u(x) \le \frac 1 2 d(p,x)^2 + C(n). 
\end{equation}

 When $(M,g)$ is closed and $\mathrm{Vol}(M,g)\le V$. We note by Jensen's inequality $a\le \log V$. The Ricci soliton $(M,g,u)$ gives rise to a Ricci flow $g(t) = \varphi_t^* g$ with initial metric $g(0) = g$, where $\varphi_t$ is the diffeomorphism group generated by $\nabla u$, $\frac{\partial g(t)}{\partial t} = - 2 Ric(g(t)) + g(t)$.  Combining with the fact that $R(g)\ge 0$ and Perelman's non-collapsing theorem, we see that $(M,g)$ is non-collapsed in the sense that if $R\le r^{-2}$ on $B_r(x)$, then $\mathrm{Vol}(B_r(x))\ge \kappa(n,A) r^n$, for all $r\in (0,\bar r(n,A)]$. With this non-collapsing and equations \eqref{eqn:u}, \eqref{eqn:grad u} and \eqref{eqn:lap u}, we can apply the same argument of Perelman as in Section 3 of \cite{ST} to show that there exists a uniform constant $C(n,A,V)>0$ such that for any closed $(M,g,u)\in \mathcal{RS}(n,A)$ with the additional assumption $\mathrm{Vol}(M,g)\le V$, 
\begin{equation}
\| u\|_{L^\infty} + \| \nabla u\|_{L^\infty(M,g)} + \| R\|_{L^\infty} + \mathrm{diam}(M,g)\le C(n,A,V).
\end{equation}
The non-collapsing of $(M,g)$ also implies a uniform lower bound on $\mathrm{Vol}(M,g)$.  Now we can apply the main theorem of \cite{Z}.

In general, when $(M,g)$ is complete, applying \cite{P} to the Ricci flow associated to $(M,g)$, there exists a $\kappa=\kappa(A,n)$ such that $(M,g)$ is $\kappa$-noncollapsed. In particular, $\mathrm{Vol}(B(p,1))\ge c(A,n) >0$. On any geodesic ball $B(p,r)$ with $p$ being the minimal point of $u$, $|\nabla u|\le \frac 1 2 r^2 + C(n,A)$. By the Cheeger-Colding theory for Bakry-Emery Ricci tensor $Ric(g) + \nabla^2 u$ (\cite{WZ,ZZh}), for any sequence of $(M_i,g_i, u_i, p_i)\in \mathcal{RS}(n, A)$ converges (up to a subsequence) in pointed Gromov-Hausdorff topology to a metric space $(M_\infty, d_\infty, p_\infty)$. Here we choose $p_i$ to be a minimum point of $u_i$. $M_\infty$ has the regular-singular decomposition $M_\infty = \mathcal R\cup \Sigma$. Recall a point $y\in\mathcal R$ if all tangent cone of $(M_\infty,d_\infty)$ at $y$ is isometric to $\mathbb R^n$. From \cite{ZZh} we know the singular set $\Sigma$ is closed and of Hausdorff dimension  at most $n-4$ and $d_\infty$ on $\mathcal R$ is induced by a $C^\alpha$ metric $g_\infty$. For any $y\in \mathcal R$ and $M_i\ni y_i\xrightarrow{GH} y$, when $i$ is large enough there exists a uniform $r_0=r_0(y)$ such that $(B_{g_i}(y_i,r_0), g_i)$ has uniform $C^\alpha$ bound (Theorem 1.2 of \cite{ZZh}). By choosing $r_0$ even smaller if possible, we may assume the isoperimetric constant of $(B_{g_i}(y_i,r_0), g_i)$ is very small so that we can apply Perelman's pseudo-locality theorem (\cite{P}) to the associated Ricci flow to derive uniform higher order estimates of $g_i$ nearby $y_i$, which in turn gives local estimates of $u_i$. So locally near $y_i$, the convergence is smooth and we conclude that the metric $g_\infty$ in a small ball around $y$ is a Ricci soliton.

\qed

\end{proof}

We remark that in the compact case, a compactness result is obtained earlier by Zhang \cite{Z} assuming a uniform upper bound for the diameter and a uniform lower bound for the volume. 

\bigskip

\noindent {\bf{Acknowledgements:}} The authors would like to thank Xiaowei Wang for valuable discussions and for teaching us his proof of  Corollary \ref{cor:2.1} in his work \cite{LWX} with Li and Xu.

\bigskip

{\noindent \footnotesize $^*$ Department of Mathematics\\
Columbia University, New York, NY 10027\\

{\noindent \footnotesize $^{**}$ Department of Mathematics\\
Columbia University, New York, NY 10027\\

\noindent $\dagger$ Department of Mathematics\\
Rutgers University, Piscataway, NJ 08854\\

\noindent $\ddagger$ Department of Mathematics\\
Rutgers University, Newark, NJ 07102\\ }

\end{document}